\documentclass[11pt,leqno]{amsart}

\usepackage{amsfonts,amssymb}
\usepackage{graphicx,tikz}
\usepackage{tikz-cd} 
\usepackage{float}
\usepackage{amsmath, amsthm, amsfonts}
\usepackage{hyperref}
\usepackage{enumitem}  
\usepackage[capitalise]{cleveref}
\usepackage{comment}
\usepackage{enumitem}
\usepackage{amsrefs}
\usepackage{nicefrac}
\usepackage[official]{eurosym}
\usepackage{graphicx}
\usepackage{nomencl}
\usepackage{amsfonts}
\usepackage{hyperref}
\usepackage{tabto}
\usepackage{amssymb}
\usepackage{fancyhdr}
\usepackage{amscd}
\usepackage[english]{babel}

\usepackage[margin=1.15in]{geometry}

\theoremstyle{plain}

\newtheorem*{theorem*}{Theorem}
\newtheorem{thm}{Theorem}[section]
\newtheorem{pr}[thm]{Proposition}
\newtheorem{cor}[thm]{Corollary}
\newtheorem{lem}[thm]{Lemma}
\newtheorem{theorem}{Theorem}
\newtheorem{corollary}[theorem]{Corollary}

\theoremstyle{definition}

\newcommand\T{{\mathcal{T}}}
\makeatletter
\def\cleardoublepage{\clearpage\if@twoside \ifodd\c@page\else
	\hbox{}
	\thispagestyle{empty}
	\newpage
	\if@twocolumn\hbox{}\newpage\fi\fi\fi}
\makeatother
\DeclareMathOperator{\Aut}{Aut}
\DeclareMathOperator{\St}{St}

\def\GG{G_{\omega}}
\numberwithin{equation}{section}

 \keywords{Groups acting on rooted trees, finite $p$-groups, ramification structures}
 \subjclass[2010]{Primary  20E08;  Secondary 20D15, 14J29}

\begin{document}
	
	\title[Ramification structures for quotients of the Grigorchuk groups]{Ramification structures for quotients of the Grigorchuk groups}

\author[M. Noce]{Marialaura Noce}
\address{Marialaura Noce: Mathematisches Institut, Georg-August Universit\"{a}t zu G\"{o}ttingen, Bunsenstr. 3-5, 37073 G\"{o}ttingen, Germany}
\email{mnoce@unisa.it}

 \author[A. Thillaisundaram]{Anitha Thillaisundaram}
 
 \address{Anitha Thillaisundaram: Centre for Mathematical Sciences, Lund University, S\"{o}lvegatan~18, 223 62 Lund, Sweden}
 \email{anitha.t@cantab.net}

\thanks{The first author is supported by EPSRC, grant number
1652316 and by the Spanish Government grant MTM2017-86802-P, partly with FEDER funds, and by the “National Group for Algebraic and Geometric Structures and their
Applications” (GNSAGA - INdAM). She also acknowledges financial support from a London Mathematical Society Joint Research Groups in the UK (Scheme 3) grant. The second author acknowledges support from  EPSRC, grant EP/T005068/1.}

\begin{abstract}
Groups associated to surfaces isogenous to a higher product of curves can be characterised by a purely group-theoretic condition, which is the existence of a so-called ramification structure.
In this paper, we prove that infinitely many quotients of the Grigorchuk groups admit ramification structures. This gives the first explicit infinite family of 3-generated finite 2-groups with ramification structures.%all ramification structures being non-Beauville .
\end{abstract}

	\maketitle
%	\tableofcontents

\section{Introduction}

Let $G$ be a finite group and let $T=(g_1,g_2,\ldots,g_r)$, for some $r\in\mathbb{N}_{\ge 3}$, be a tuple of non-trivial elements of~$G$.
The tuple~$T$ is called a \emph{(spherical) system of generators} of~$G$ if 
$\langle g_1,\ldots,g_r\rangle =G$ and $g_1g_2\cdots g_r=e$, where $e$ denotes the identity element of the group.
 Next, let $\Sigma(T)$ be the union of all conjugates of the cyclic subgroups generated by the elements of~$T$, that is,
\[
\Sigma(T)=\bigcup_{g\in G} \bigcup_{i=1}^r \langle g_i\rangle^g.
\]
We say that two tuples~$T_1$ and~$T_2$ are \emph{disjoint} if $\Sigma(T_1)\cap \Sigma(T_2)=\{e\}$.

An \emph{(unmixed) ramification structure} of size $(r_1,r_2)$ for a finite group~$G$, as defined in~\cite{BCG2}, is a pair $(T_1,T_2)$ of disjoint systems of generators of~$G$, where $|T_1|=r_1$ and $|T_2|=r_2$. Any finite group~$G$ with a ramification structure gives rise to a surface isogenous to a higher product of curves; see~\cite{Catanese}.
 Recall that an algebraic surface $S$ is \emph{isogenous to a higher product of curves} if it is isomorphic to $(C_1\times C_2)/G$, where $C_1$ and $C_2$ are curves of genus at least~$2$, and $G$ is a finite group acting freely on $C_1\times C_2$. Such groups~$G$ that are associated to surfaces isogenous to a higher product of curves are characterised by the existence of a ramification structure.

We are interested only in the case when $r_1=r_2$, so we will simply say \emph{ramification structure} in the sequel.   When $r_1=r_2=3$, one uses the term \emph{Beauville structure} instead of ramification structure. Also, a group admitting a Beauville structure is called a \emph{Beauville group}.

The theory of groups admitting ramification structures has only recently gained attention, with ramification structures being a natural generalisation of the more well-studied Beauville structures. Hence this subject of groups with ramification structures is  closely related to the established topic  of Beauville surfaces and Beauville groups. Originally,  groups with ramification structures were studied in the context of classifying surfaces of complex dimension two isogenous to a product of two curves. These surfaces can be viewed as Riemannian surfaces of real dimension two. In particular, in the definition of ramification structures the parameters $r_1$ and $r_2$  denote  the number of sides of the polygons in the tessellations induced by the group.

Surfaces isogenous to a product of curves were first considered by Catanese in~\cite{Catanese}, and their connection with finite groups was developed by Bauer, Catanese and Grunewald in~\cite{BCG1, BCG3, BCG2}. As mentioned before, there has been a lot of work done concerning Beauville groups; see surveys \cites{ BBF, fai, fai2, jon}, with recent work including~\cites{Fairbairn, Gul2, CF}. In particular,   finite non-abelian simple groups (excluding $A_5$) are known to be Beauville groups (see~\cite{FMP,GM}), and finding finite $p$-groups that are Beauville has received special attention (see \cite{BBF,BBPV2,SV,FAG,GUA, FAGV}). We note that a Beauville group admits ramification structures of size $(r_1,r_2)$ for any $r_1,r_2\ge 3$; see~\cite[Lem.~2.9]{gul}. Beauville groups are necessarily 2-generated.
However, the study of ramification structures in $d$-generator groups, for $d>2$, is very much in its infancy. 
In~\cite{GP2}, Garion and Penegini characterised the abelian groups with ramification structures, and in~\cite{BBPV}, a finite family of $3$-groups with ramification structures was established.
The study of non-abelian nilpotent groups admitting ramification structures that are not Beauville  was initiated in~\cite{gul}; there, finite nilpotent groups possessing a certain nice power structure were considered, and a characterisation for such groups admitting ramification structures was given. 

Groups acting on regular rooted trees were first seen as a source for finite $p$-groups with ramification structures in the work of
 G\"{u}l and Uria-Albizuri~\cite{GUA}. They  showed that quotients of a Grigorchuk-Gupta-Sidki group acting on the $p$-adic tree, for $p$ an odd prime, admit ramification structures if and only if the Grigorchuk-Gupta-Sidki group is periodic.
 The Grigorchuk-Gupta-Sidki (GGS-)groups  were some of the early examples of groups acting on rooted trees, and were also some of the  easily describable examples of infinite finitely generated periodic groups that were constructed in the 1980s. Since then, groups
 acting on regular rooted trees have provided many other interesting and exotic examples, such as infinite finitely generated groups of intermediate growth, infinite finitely generated amenable but not elementary amenable groups, etc.

The most well-known group acting on a regular rooted tree %that was constructed 
is the (first) Grigorchuk group~\cite{GR}. The Grigorchuk group acts on the binary rooted tree~$\T$ and is a 3-generated infinite periodic group with many interesting properties. The Grigorchuk group was generalised in~\cite{Grigorchuk} to a family of Grigorchuk groups~$G_\omega$, for $\omega$ an infinite sequence in $\{0,1,2\}$, which also act on the binary rooted tree. Each group $\GG$ is generated by one rooted automorphism~$a$ which swaps the two maximal subtrees, and two automorphisms $b_\omega$ and $c_\omega$, both of which are defined recursively according to $\omega$ and stabilise the rightmost infinite ray of the tree; see Section~2 for the precise definitions. 

A Grigorchuk group~$\GG$ is periodic if and only if the sequence~$\omega$ has $0,1,2$ appearing infinitely many times. Apart from Grigorchuk's original paper~\cite{Grigorchuk} and \cites{Drum,GZ,MP1,MP2, Petrides}, the non-periodic Grigorchuk groups appear to have received slightly less attention, as compared to the periodic ones; see~\cites{Leonov,Pervova3,Pervova,erschler}. In particular, in~\cite{Pervova} it is proved that the periodic Grigorchuk groups have the congruence subgroup property, but as we will see, this result generalises to all Grigorchuk groups~$\GG$ with $\omega$ not eventually constant; cf. Theorem~\ref{thm:CSP}. Here a group~$G$ acting on a rooted tree is said to have the \emph{congruence subgroup property} if every finite-index subgroup
contains the pointwise stabiliser $\St_G(n)$ of the vertices at some level~$n\in\mathbb{N}$ of the tree. This is parallel to the classical congruence subgroup property
 for subgroups of $\text{SL}(n, \mathbb{Z})$ for $n > 2$. 
 
 Therefore, in light of the congruence subgroup property, the natural quotients of the Grigorchuk groups~$\GG$ to consider are the quotients by the level stabilisers.
In this paper, we show that all but finitely many such quotients admit ramification structures. %extend the result of~\cite{GUA} to  the Grigorchuk groups~$G_\omega$. 
Before stating our result, we need the following notation. Let $\Omega$ be the set of all infinite sequences in $\{0,1,2\}$. For $\omega\in \Omega$ and $k\in\{0,1,2\}$, we define
\[
i_k(\omega)=\min\{n\in\mathbb{N} \mid \omega_n=k\}
\]
if the minimum exists, otherwise  $i_k(\omega)=\infty$.
Let 
\[
N= \max\{i_0(\sigma \omega), i_1(\sigma \omega),i_2(\sigma \omega)\} +4,
\]
and if $i_k(\sigma \omega)=\infty$ for exactly one $k\in\{0,1,2\}$, then we set
\[
\widetilde{N}=\max\{i_j(\sigma \omega) \mid j\ne k\}+4.
\]

\begin{theorem}\label{thm:main}
Let $G:=G_\omega$ be a Grigorchuk group acting on the binary rooted tree. Then the quotient $G/\St_{G}(n)$ admits a ramification structure for all $n\ge M$, where
\[
M=\begin{cases}
\,\,N & \quad\text{if $N$ is finite};\\
\,\,\widetilde{N}  & \quad\text{if $N$ is infinite and $\sigma\omega$ is not the constant sequence};\\
\,\,4 & \quad\text{if $\sigma\omega$ is the constant sequence}.
\end{cases}
\]
%\begin{enumerate}
%\item if $N$ is finite, {\color{blue}the} quotient $\GG/\St_{\GG}(n)$ admits a ramification structure for all $n\ge N$;
%\item  if $N$ is infinite and $\sigma\omega$ is not the constant sequence,  {\color{blue}the} quotient $\GG/\St_{\GG}(n)$ admits a ramification structure for all $n\ge \widetilde{N}$;
%\item if $\sigma\omega$ is the constant sequence,  {\color{blue}the} quotient $\GG/\St_{\GG}(n)$ admits a ramification structure for all $n\ge 4$.
%\end{enumerate}
\end{theorem}

In particular, our result yields the first explicit infinite family of 3-generated finite $2$-groups with ramification structures none of which are Beauville. We also note that the result of~\cite{GUA} has been extended to two distinct generalisations of the GGS-groups: in~\cite{DGT1} it was shown that quotients of periodic GGS-groups acting on the $p^n$-adic tree, for $p$ any prime and $n\in\mathbb{N}$ admit Beauville structures, and in~\cite{DGT2}, quotients of periodic multi-EGS groups acting on the $p$-adic tree, for $p$ an odd prime, were shown to admit ramification structures that are not Beauville.

For $G$ a finite $p$-group of exponent~$p^e$, recall that $G$ is said to be \emph{semi-$p^{e-1}$-abelian} if, for every $x,y\in G$,
\[%begin{equation}\label{eq: def semi-p^e abelian}
x^{p^{e-1}}=y^{p^{e-1}} \quad\text{if and only if}\quad (xy^{-1})^{p^{e-1}}=1.
\]%end{equation}
The following result is an immediate consequence of~\cite[Thm.~A]{gul}:

\begin{corollary}
Let $G:=G_\omega$ be a Grigorchuk group acting on the binary rooted tree. Then, for $n\ge M$ as defined in Theorem~\ref{thm:main}, the quotient group $G/\St_{G}(n)$ is not semi-$p^{e-1}$-abelian, where $p^e$ is the exponent of the group $G/\St_{G}(n)$.
\end{corollary}

\subsection*{Acknowledgements} We thank all referees for suggesting valuable improvements to the exposition of the paper.

%%%%%

\section{Preliminaries}
Let $\T$ be the binary rooted tree,
  meaning that there is a distinguished vertex called the root and all vertices have $2$ children. Everything in this section can be defined for more general rooted trees, but we will not consider this general setting here, and instead refer the reader to~\cite{Handbook} or~\cite{NewHorizons}.
  
  Using the
  alphabet $X = \{0,1\}$, the vertices $u_\nu$ of~$\T$ are
  labelled bijectively by elements $\nu$ of the free
  monoid~$X^*$ and $\T$ is constructed from~$X^*$ as follows: the root of~$\T$
  is labelled by the empty word~$\varnothing$, and for each word $\nu \in X^*$ and letter $x \in X$ there is an edge
  connecting $u_\nu$ to~$u_{\nu x}$.  We say
  that $u_\nu$ precedes $u_\mu$ whenever $\nu$ is a prefix of $\mu$.

  We recall that there is a natural length function  on~$X^*$.  The words
  $\nu$ of length $\lvert \nu \rvert = n$, representing vertices
  $u_\nu$ that are at distance $n$ from the root, are the $n$th
  level vertices and form the \textit{$n$th layer} of the tree. The elements of the boundary $\partial \T$ correspond
  naturally to infinite simple rooted paths, and are in one-to-one correspondence
  with the $2$-adic integers.

  We denote by $\T_u$ the full rooted subtree of~$\T$ that has its root at
  a vertex~$u$ and includes all vertices succeeding~$u$.  For any
  two vertices $u = u_\nu$ and $v = u_\mu$, the map
  $u_{\nu \tau} \mapsto u_{\mu \tau}$, induced by replacing the
  prefix $\nu$ by $\mu$, yields an isomorphism between the
  subtrees $\T_u$ and~$\T_v$.  

  Now every automorphism of $\T$ fixes the root and the orbits of
  $\mathrm{Aut}(\T)$ on the vertices of the tree~$\T$ are precisely its
  layers. For $f \in \mathrm{Aut}(\T)$, the image of a vertex $u$ under
  $f$ is denoted by~$u^f$.  Observe that $f$ induces a faithful action
  on~$X^*$ such that
  $(u_\nu)^f = u_{\nu^f}$.  For $\nu \in X^*$ and
  $x \in X$ we have $(\nu x)^f = \nu^f x'$ where $x' \in X$ is
  uniquely determined by $\nu$ and~$f$.  This induces a permutation
  $f_\nu$ of $X$ so that
  \[
  (\nu x)^f = \nu^f x^{f_\nu}, \quad \text{and hence}
  \quad   (u_{\nu x})^f = u_{\nu^f x^{f_\nu}}.
  \]
The permutation~$f_\nu$ is called the \textit{label} of~$f$ at~$\nu$, and   the  collection  of  all  labels  of~$f$ constitutes the \emph{portrait} of~$f$. There is a one-to-one correspondence between automorphisms of~$\T$ and portraits. We say that the automorphism $f$ is \textit{rooted} if $f_\nu = 1$ for
  $\nu \ne \varnothing$, and we say that $f$ is \textit{directed}, with directed path $\ell \in \partial \T$, if the support $\{ \nu \mid f_\nu \ne 1 \}$ of its labelling is infinite and marks only vertices at distance $1$ from the set of vertices corresponding  to the path~$\ell$. 

When convenient, we do not differentiate between $X^*$ and vertices of~$\T$. In this spirit we define the \textit{section} of~$f$ at a vertex~$u$ to be the unique automorphism $f(u)$ of $\T \cong \T_{|u|}$ given by the condition $(uv)^f = u^f v^{f(u)}$ for $v \in X^*$.

%%%

%\subsection{Subgroups of $\Aut(\T)$}
Next, for $G\le \Aut(\T)$, we define the
\textit{vertex stabiliser} $\St_G(u)$ to be the subgroup
consisting of elements in $G$ that fix the vertex~$u$. For
$n \in \mathbb{N}$, the \textit{$n$th level stabiliser}
  $\St_G(n)= \bigcap_{\lvert \nu \rvert =n}
  \St_G(u_\nu)$
is the subgroup consisting of automorphisms that fix all vertices at
level~$n$.  Denoting by $\T_{[n]}$ the finite rooted subtree of~$\T$ of all vertices up to level~$n$, we see that $\St_G(n)$ is equal to
the kernel of the induced action of~$G$ on $\T_{[n]}$.

Now for $n\in\mathbb{N}$, each $g\in \St_{\mathrm{Aut}(\T)} (n)$ can be 
  described completely in terms of its restrictions $g_1,\ldots,g_{2^n}$ to the subtrees
  rooted at vertices at level~$n$.  Indeed, there is a natural
  isomorphism
\[
\psi_n \colon \St_{\mathrm{Aut}(\T)}(n) \rightarrow
\prod\nolimits_{\lvert \nu \rvert = n} \mathrm{Aut}(\T_{u_\nu})
\cong \mathrm{Aut}(\T) \times \overset{2^n}{\dots} \times
\mathrm{Aut}(\T)
\]
given by
\[
g \longmapsto (g_1,\ldots,g_{2^n}).
\]
The restrictions $g_1,\ldots, g_{2^n}$ are actually the sections $g(u)$, for the $n$th level vertices~$u$. Using the identification of~$X^*$ with $\T$, the $n$th level vertices are ordered  lexicographically. Then in the image of~$\psi_n$, the sections $g_1,\ldots,g_{2^n}$ are ordered accordingly.
For ease of notation, we write $\psi=\psi_1$.

For  $\nu\in X^*$, we further define 
\[
\varphi_\nu :\mathrm{St}_{\mathrm{Aut}(\T)}(u_\nu) \rightarrow \mathrm{Aut}(\T_{u_\nu}) \cong \mathrm{Aut}(\T)
\]
to be the natural restriction of $f\in \St_{\Aut(\T)}(u_\nu)$ to the section~$f(u_\nu)$.

For a vertex~$u$ of~$\T$ and $G\le \Aut(\T)$, the \textit{rigid vertex stabiliser} $\mathrm{Rist}_G(u)$ of~$u$ in~$G$ is the subgroup
 consisting of all automorphisms in $G$ that fix
all vertices~$v$ of~$\T$ not succeeding~$u$.
 For $n\in\mathbb{N}$, the \textit{rigid $n$th level stabiliser} is the direct product of the rigid vertex stabilisers of the vertices at level~$n$:
  \[
  \mathrm{Rist}_G(n) = \prod\nolimits_{\lvert \nu \rvert = n}
  \mathrm{Rist}_G(u_\nu) \trianglelefteq G.
  \]

We recall that a level-transitive group $G \le \Aut(\T)$ (i.e.\ a group acting transitively at each level of the tree)  is a \emph{branch
  group}, if $\mathrm{Rist}_G(n)$ has finite index in $G$ for every $n \in \mathbb{N}$. %; and $G$ is \emph{weakly branch}, if $\mathrm{Rist}_G(n)$ is non-trivial for every $n \in \mathbb{N}$.

\section{The Grigorchuk groups and their properties}
The family of groups $G_{\omega}$ acts on~$\T$ and is defined via the parameter $\omega=(\omega_1, \omega_2, \dots)$, such that  $\omega_i \in \{0,1,2\}$ for any $i\in\mathbb{N}$. We denote by $\Omega$ the set of all such $\omega$ and for $n\in\mathbb{N}$, we write $\sigma^n \omega$ for the sequence $(\omega_{n+1}, \omega_{n+2}, \dots)$. 

For $\omega \in \Omega$, the Grigorchuk group~$G_{\omega}$ is generated by $a, b_{\omega}, c_{\omega}, d_{\omega}$, where $a$ is the rooted automorphism corresponding to the cycle $(1\, 2)$, and the directed automorphisms $ b_{\omega}, c_{\omega}, d_{\omega}$ lie in the first level stabiliser and are defined as follows. 
For $n\in\mathbb{N}\cup \{0\}$, let $u_n=1\,\overset{n}\ldots\, 1\in X^*$. Then 
$b_{\omega}(u_n0)$ is trivial if $\omega_{n+1} = 2$ and is equal to $a$ otherwise; the automorphism~$c_{\omega}(u_n0)$ is trivial if $\omega_{n+1} = 1$ and is equal to $a$ otherwise; and finally $d_{\omega}(u_n0)$ is trivial if $\omega_{n+1} = 0$ and is equal to $a$ otherwise. More precisely: 
$$
\varphi_u(b_{\omega})=
\begin{cases}
a &\quad \mbox{ if } u=u_n0 \mbox{ for some } n \mbox{ and } \omega_{n+1} \neq 2,\\
e &\quad \mbox{ if } u=u_n0 \mbox{ for some } n \mbox{ and } \omega_{n+1} = 2,\\
b_{\sigma^n\omega} &\quad \mbox{ if } u=u_n\text{ for some }n;
\end{cases}
$$
$$
\varphi_u(c_{\omega})=
\begin{cases}
a &\quad \mbox{ if } u=u_n0 \mbox{ for some } n \mbox{ and } \omega_{n+1} \neq 1,\\
e &\quad \mbox{ if } u=u_n0 \mbox{ for some } n \mbox{ and } \omega_{n+1} = 1,\\
c_{\sigma^n\omega} &\quad \mbox{ if } u=u_n \mbox{ for some } n;
\end{cases}
$$
$$
\varphi_u(d_{\omega})=
\begin{cases}
a &\quad \mbox{ if } u=u_n0 \mbox{ for some } n \mbox{ and } \omega_{n+1} \neq 0,\\
e &\quad \mbox{ if } u=u_n0 \mbox{ for some } n \mbox{ and } \omega_{n+1} = 0,\\
d_{\sigma^n\omega} &\quad \mbox{ if } u=u_n \mbox{ for some } n.
\end{cases}
$$

For example, if we take the sequence $(0,\omega_2, \omega_3, \dots)$, we obtain
\[
\psi(b_{\omega})  = (a, b_{\sigma\omega}), \quad
\psi(c_{\omega}) = (a, c_{\sigma\omega}),\quad\text{and}\quad
\psi(d_{\omega}) = (e, d_{\sigma\omega}).
\]
If $\omega=(0,1,2, 0, 1, 2, \dots)$, then $G_{\omega}$ is the first Grigorchuk group.

Notice that for any choice of the sequence~$\omega$, apart from the three constant sequences, all these generators are elements of order~$2$. Furthermore the generators $b_\omega$, $c_\omega$ and $d_\omega$ commute with each other, and we observe the relation $b_\omega c_\omega =d_\omega$ in~$\GG$, which we will use in the sequel without special mention. 

Many properties of~$\GG$ depend on the choice of the sequence~$\omega$. Let $\Omega_0$ denote the set of $\omega\in\Omega$  such that $0$, $1$ and $2$ occur infinitely many times in~$\omega$. The group~$\GG$ is periodic if and only if $\omega\in\Omega_0$, and $\GG$ is branch if and only if $\omega$ is not eventually constant; see~\cite{Grigorchuk}. Moreover, if $\omega$ is eventually
constant then $G_\omega$ is of polynomial growth (as it is virtually abelian), otherwise $G_\omega$ gives a wealth of examples of groups of intermediate growth; see~\cite{Grigorchuk} and~\cite{erschler}. If~$\omega$ is a constant sequence, then $\GG$ is isomorphic to~$D_{\infty}$, the 2-generated infinite dihedral group.

Throughout the paper we will always exclude the case when $\omega$ is a constant sequence. In other words, we redefine $\Omega$ to be the aforementioned set of sequences~$\omega$ apart from the three constant sequences. To determine the ramification structures for quotients of the Grigorchuk groups, we need some preliminary tools. Namely, we determine the orders of specific elements of~$\GG$ for certain $\omega\in \Omega$. We denote by~$o(x)$ the order of an element $x$, and for $\omega\in \Omega$ and $k\in\{0,1,2\}$, we recall that
\[
i_k(\omega)=\min\{n\in\mathbb{N} \mid \omega_n=k\}
\]
and $i_k(\omega)=\infty$ if no minimum exists.
For convenience, we often write $i_0,i_1,i_2$ when $\omega$ is clear.

\begin{lem}\label{lem: order elements}
Let $G_{\omega}$ be a Grigorchuk group for $\omega\in \Omega$.
Then, using the above notation, if~$i_0, i_1, i_2$ are finite, we have
\[
o(ab_{\omega})=2^{i_2+\varepsilon_2}, \quad o(ac_{\omega})=2^{i_1+\varepsilon_1},\quad  o(ad_{\omega})=2^{i_0+\varepsilon_0},
\]
where $\varepsilon_k=0$ if $\sigma^{i_k}\omega=kkk\cdots$ and $\varepsilon_k=1$ otherwise, for $k\in\{0,1,2\}$.

If furthermore $i_0(\sigma \omega), i_1(\sigma\omega), i_2(\sigma\omega)$ are finite, then
\begin{align*}
o(ad_{\omega}ab_{\omega})=o(ab_{\omega}ad_{\omega})=&
\begin{cases}
o(ad_{\sigma\omega}) %2^{i_0(\sigma\omega)+1} 
& \text{ if }\omega_1=0,\\
\max\{o(ad_{\sigma\omega}), o(ab_{\sigma\omega})\}
%2^{\max\{i_0(\sigma\omega),i_2(\sigma\omega)\}+1} 
& \text{ if }\omega_1=1,\\
o(ab_{\sigma\omega})
%2^{i_2(\sigma\omega)+1} 
& \text{ if }\omega_1=2;
\end{cases}\\
o(ad_{\omega}ac_{\omega})=o(ac_{\omega}ad_{\omega})=&
\begin{cases}
o(ad_{\sigma\omega})
%2^{i_0(\sigma\omega)+1} 
& \text{ if }\omega_1=0,\\
o(ac_{\sigma\omega})
%2^{i_1(\sigma\omega)+1} 
& \text{ if }\omega_1=1,\\
\max\{o(ad_{\sigma\omega}), o(ac_{\sigma\omega})\}
%2^{\max\{i_0(\sigma\omega),i_1(\sigma\omega)\}+1} 
& \text{ if }\omega_1=2;
\end{cases}\\
o(ac_{\omega}ab_{\omega})=o(ab_{\omega}ac_{\omega})=&
\begin{cases}
\max\{o(ac_{\sigma\omega}), o(ab_{\sigma\omega})\}
%2^{\max\{i_1(\sigma\omega),i_2(\sigma\omega)\}+1} 
& \text{ if }\omega_1=0,\\
o(ac_{\sigma\omega})
%2^{i_1(\sigma\omega)+1} 
& \text{ if }\omega_1=1,\\
o(ab_{\sigma\omega})
%2^{i_2(\sigma\omega)+1} 
& \text{ if }\omega_1=2.
\end{cases}
\end{align*}
\end{lem}

\begin{proof}
Consider the element $ab_\omega$. We see that
the components of $\psi_n((ab_\omega)^{2^n})$ are either $ab_{\sigma^n\omega}$ or $b_{\sigma^n\omega}a$, for $1\le n<i_2$, so
$\psi_{i_2-1}((ab_\omega)^{2^{i_2-1}})=(x_1,\ldots,x_{2^{i_2-1}})$
where $x_j\in\{ab_{\sigma^{i_2-1}\omega},b_{\sigma^{i_2-1}\omega}a\}$ for $j\in\{1,\ldots, 2^{i_2-1}\}$. 

If $\sigma^{i_2}\omega=22\cdots$, then $b_{\sigma^{i_2-1}\omega}$ is the trivial automorphism. Hence $\psi_{i_2-1}((ab_\omega)^{2^{i_2-1}})=(a,\overset{2^{i_2-1}}\ldots,a)$ and thus $o(ab_\omega)=2^{i_2}$.

If $\sigma^{i_2}\omega\ne 22\cdots$, then $
\psi_{i_2}((ab_\omega)^{2^{i_2}})=(b_{\sigma^{i_2}\omega},\,\overset{2^{i_2}}\ldots\,,b_{\sigma^{i_2}\omega})$
and as  $o(b_{\sigma^{i_2}\omega})=2$, it follows that $o(ab_\omega)=2^{i_2+1}$. 

Similarly for $ac_\omega$ and $ad_\omega$.

Lastly, for $ad_{\omega}ab_{\omega}=((ab_{\omega}ad_{\omega})^a)^{-1}$, we note that 
\[
\psi(ad_{\omega}ab_{\omega})=
\begin{cases}
(d_{\sigma\omega}a, b_{\sigma\omega}) & \text{ if }\omega_1=0,\\
(d_{\sigma\omega}a, ab_{\sigma\omega}) & \text{ if }\omega_1=1,\\
(d_{\sigma\omega}, ab_{\sigma\omega}) & \text{ if }\omega_1=2.
\end{cases}
\]
Thus the result follows, and similarly for the remaining cases.
\end{proof}

To end this section, we collect together a few observations. The next result follows from \cite[Prop.~1]{Pervova_AT}.
\begin{lem}\label{lem:abelianisation}
Let $G_{\omega}$ be a Grigorchuk group for $\omega\in \Omega$. Then 
\[
\GG/\GG'\cong\langle a\rangle \times \langle b_\omega, c_\omega\rangle \cong \langle a\rangle \times \langle b_\omega, d_\omega\rangle\cong \langle a\rangle \times \langle c_\omega, d_\omega\rangle  \cong C_2\times C_2\times C_2.
\]
\end{lem}

%The next result shows that many Grigorchuk groups exhibit a form of self-replication.

%\begin{lem}\label{lem:fractal}
%Let $G_{\omega}$ be a Grigorchuk group for $\omega\in \Omega$. 
%Then for every $n\in\mathbb{N}$ and a vertex~$u$ of level~$n$, we have $\St_{\GG}(u)=G_{\sigma^n\omega}$. 
%\end{lem}

%\begin{proof}
% Since $\GG$ is level-transitive, it suffices to establish the result for $u=1$. The result is immediate from $\varphi_1(b_\omega)$, $\varphi_1(c_\omega)$, $\varphi_1(b_\omega^a)$ and $\varphi_1(c_\omega^a)$.
%\end{proof}

We recall the following definitions from~\cite{Pervova}. For any given $\omega\in\Omega$, in the group~$\GG$ exactly one of the generators $b_\omega$, $c_\omega$, and $d_\omega$ has trivial label at the vertex $0\in X$. This generator will be called the \emph{$d$-generator} of~$\GG$. Next we let
\[
m(\omega)=\max\{n\in \mathbb{N}\mid \omega_1=\cdots=\omega_n\}.
\]
Recall that we always exclude constant sequences, and therefore $m(\omega)$ is always finite. Let $y_\omega$ be the generator of~$\GG$ such that $y_{\sigma^{m(\omega)}\omega}$ is the $d$-generator of $G_{\sigma^{m(\omega)}\omega}$. We call $y_\omega$ the \emph{$c$-generator} of~$\GG$. 

The four results that we collect below 
were  proved in~\cite[Cor.~1.1, Prop.~2.1, Lem.~3.1 and Prop.~3.1]{Pervova} for periodic Grigorchuk groups, but it can be checked that the proofs hold more generally for all Grigorchuk groups, with natural and minor modifications for the case when $\sigma^{m(\omega)}\omega$ is a constant sequence.

\begin{cor}\label{cor:abelianisation}
Let $G_{\omega}$ be a Grigorchuk group for $\omega\in \Omega$. Write  $x_\omega$, respectively $y_\omega$, for the $d$-generator, respectively $c$-generator, of~$\GG$, and let $N=\langle x_\omega\rangle^{G_\omega}$ be the normal closure of~$x_\omega$ in~$\GG$. Then
$\GG/N \cong \langle a,y_\omega\rangle$.
\end{cor}

\begin{pr}\label{pr:rist}
Let $G_{\omega}$ be a Grigorchuk group for $\omega\in \Omega$ and let $x_\omega$ be the $d$-generator of~$\GG$.
For $n\in\mathbb{N}$, let $z_{\sigma^{n-1}\omega}$ be the $d$-generator of~$G_{\sigma^{n-1}\omega}$, and let $N_{\sigma^n\omega}$ denote the normal closure of $(az_{\sigma^n\omega})^2$ in~$G_{\sigma^n\omega}$. Then for any vertex~$u$ of length~$n$, 
\[
\varphi_u(\textup{Rist}_{\GG}(u))=\begin{cases}
\langle x_{\sigma^n\omega}\rangle^{G_{\sigma^n\omega}} N_{\sigma^n\omega} & \text{ if }1\le n\le m(\omega),\\
N_{\sigma^n\omega} & \text{ if }n>m(\omega).
\end{cases}
\]
\end{pr}

\begin{lem}\label{lem:derived-stab}
Let $G_{\omega}$ be a Grigorchuk group for $\omega\in \Omega$. Then $\St_{\GG}(m(\omega)+2)\le \GG'$.
\end{lem}

Recall that we say a subgroup~$H$ of~$\GG$ is a \emph{congruence subgroup} if $H$ contains $\St_{\GG}(n)$ for some $n\in\mathbb{N}$.

\begin{pr}\label{pr:congruence-subgroup}
Let $G_{\omega}$ be a Grigorchuk group for $\omega\in \Omega$. Then for each $x_\omega\in\{b_\omega, c_\omega, d_\omega\}$, the normal closure of~$(x_\omega a)^2$ is a congruence subgroup.
\end{pr}

We recall that a Grigorchuk group~$\GG$ is branch if and only if $\omega$ is not eventually constant.

\begin{thm}
\label{thm:CSP}
Let $G_{\omega}$ be a branch Grigorchuk group for $\omega\in \Omega$. Then $\GG$ has the congruence subgroup property and $\GG$ is just infinite.
\end{thm}

\begin{proof}
The second part follows from~\cite[Thm.~4.4.4]{Francoeur}, where it was shown that finitely generated branch groups with the congruence subgroup property are always just infinite. To prove the first statement, let $H$ be a finite-index subgroup. By~\cite[Cor.~2.3]{BSZ}, for some $n\in\mathbb{N}$ we have $\text{Rist}_{\GG}(n)'\le H$. Using the notation of Proposition~\ref{pr:rist}, let $z_{\sigma^{n-1}\omega}$ be the $d$-generator of~$G_{\sigma^{n-1}\omega}$. Further let $y_{\sigma^{n-1}\omega}$ be the $c$-generator of~$G_{\sigma^{n-1}\omega}$. 

%{\color{blue}
%\underline{Case 1:}  Suppose $\psi(z_{\sigma^n\omega})=(a,z_{\sigma^{n+1}\omega})$. Then $\psi(y_{\sigma^n\omega})=(e,y_{\sigma^{n+1}\omega})$, and we recall from Proposition~\ref{pr:rist} that $\psi^{-1}((z_{\sigma^{n+1}\omega}a,az_{\sigma^{n+1}\omega}))=(az_{\sigma^n\omega})^2\in\varphi_v(\text{Rist}_{\GG}(v))$, for any $n$th-level vertex~$v$.

%\textit{Subcase (a):} Suppose further that $\psi(z_{\sigma^{n+1}\omega})=(a,z_{\sigma^{n+2}\omega})$ and $\psi(y_{\sigma^{n+1}\omega})=(e,y_{\sigma^{n+2}\omega})$. As $(ay_{\sigma^{n+1}\omega})^2\in\varphi_u(\text{Rist}_{\GG}(u))$, for any $(n+1)$st-level vertex~$u$, we have that

If $\psi(z_{\sigma^n\omega})=(a,z_{\sigma^{n+1}\omega})$ then $\psi(y_{\sigma^n\omega})=(e,y_{\sigma^{n+1}\omega})$ and $\psi_{n+1}(\text{Rist}_{\GG}(n+1))\ge 1\times \overset{2^{n+1}-1}\ldots \times1\times N_{\sigma^{n+1}\omega}$ by Proposition~\ref{pr:rist}, where $N_{\sigma^{n+1}\omega}$ is the normal closure of~$(ay_{\sigma^{n+1}\omega})^2$. Therefore
$\big[(az_{\sigma^n\omega})^2, \psi^{-1}\big((e, (a y_{\sigma^{n+1}\omega}a y_{\sigma^{n+1}\omega})^g)\big)\big]\in N_{\sigma^n\omega}'= \varphi_v(\text{Rist}_{\GG}(v)')$ for any element $g\in G_{\sigma^{n+1}\omega}$ and $n$th-level vertex~$v$.  In order to deduce that  $H$ is a congruence subgroup, by Proposition~\ref{pr:congruence-subgroup}  it suffices to show that  $\psi_{k}^{-1}(1\times \overset{2^{k}-1}\ldots \times1\times N_{\sigma^{k}\omega}) \le \text{Rist}_{\GG}(n)'$ for some $k >n$.

\textit{Subcase (a):} If $\psi(z_{\sigma^{n+1}\omega})=(a,z_{\sigma^{n+2}\omega})$ and $\psi(y_{\sigma^{n+1}\omega})=(e,y_{\sigma^{n+2}\omega})$, 
as
\[
\psi\Big(\big[(az_{\sigma^n\omega})^2, \psi^{-1}\big((e, (a y_{\sigma^{n+1}\omega} a y_{\sigma^{n+1}\omega})^{z_{\sigma^{n+1}\omega}})\big)\big]\Big) =\Big(e, \big[az_{\sigma^{n+1}\omega},((a y_{\sigma^{n+1}\omega})^2)^{z_{\sigma^{n+1}\omega}}\big]\Big)
\]
and
\[
\big[az_{\sigma^{n+1}\omega},((a y_{\sigma^{n+1}\omega})^2)^{z_{\sigma^{n+1}\omega}}\big]=\psi^{-1}\Big(\big(e,((ay_{\sigma^{n+2}\omega})^2)^{z_{\sigma^{n+2}\omega}}\big)\Big),
%\varphi_u(\text{Rist}_{\GG}(u)')
\]
we obtain
\[
\psi_2^{-1}\Big(\big(e,e,e,(ay_{\sigma^{n+2}\omega})^2\big) \Big)\in \varphi_v(\text{Rist}_{\GG}(v)') 
\]
for an $n$th-level vertex~$v$.

%for any $(n+1)$st-level vertex~$u$.

\textit{Subcase (b):} If $\psi(z_{\sigma^{n+1}\omega})=(a,z_{\sigma^{n+2}\omega})$ and $\psi(y_{\sigma^{n+1}\omega})=(a,y_{\sigma^{n+2}\omega})$, writing $x_{\sigma^{n+1}\omega}=y_{\sigma^{n+1}\omega}z_{\sigma^{n+1}\omega}$, we consider
\begin{align*}
\Big( \big[\, \psi(az_{\sigma^n\omega}az_{\sigma^n\omega}),  &\,\big(e, (ay_{\sigma^{n+1}\omega}ay_{\sigma^{n+1}\omega})^{x_{\sigma^{n+1}\omega}}\big) \,\big]^{\psi(z_{\sigma^n\omega})} \Big)\\ &\qquad=\Big( e,\big[az_{\sigma^{n+1}\omega}, (ay_{\sigma^{n+1}\omega} ay_{\sigma^{n+1}\omega})^{x_{\sigma^{n+1}\omega}} \big]^{z_{\sigma^{n+1}\omega}} \Big)\\
&\qquad=\Big(e, \big( z_{\sigma^{n+1}\omega}^{\,a} y_{\sigma^{n+1}\omega}x_{\sigma^{n+1}\omega}^{\,a} y_{\sigma^{n+1}\omega} y_{\sigma^{n+1}\omega}^{\,a}\big)\Big)\\
&\qquad=\Big(e, \psi^{-1}\big(((ax_{\sigma^{n+2}\omega}ax_{\sigma^{n+2}\omega})^{z_{\sigma^{n+2}\omega}},e)\big)\Big).
\end{align*}

\textit{Subcase (c):} If $\psi(z_{\sigma^{n+1}\omega})=(e,z_{\sigma^{n+2}\omega})$ and $\psi(y_{\sigma^{n+1}\omega})=(a,y_{\sigma^{n+2}\omega})$, we recall  that 
$$\psi^{-1}\big(((z_{\sigma^{n+1}\omega}a)^2\,,\,(az_{\sigma^{n+1}\omega})^2)\big)=(az_{\sigma^n\omega})^4\in\varphi_v(\text{Rist}_{\GG}(v)),
$$
for any $n$th-level vertex~$v$, and so
\begin{equation}\label{eq:only-z}
\psi_2^{-1}\big((z_{\sigma^{n+2}\omega},z_{\sigma^{n+2}\omega},z_{\sigma^{n+2}\omega},z_{\sigma^{n+2}\omega})\big)\in\varphi_v(\text{Rist}_{\GG}(v)).
\end{equation}
Without loss of generality, we may assume that $\psi(z_{\sigma^{n+2}\omega})=(a,z_{\sigma^{n+3}\omega})$ and $\psi(y_{\sigma^{n+2}\omega})=(e,y_{\sigma^{n+3}\omega})$. Indeed, if $\psi(z_{\sigma^{n+2}\omega})=(e,z_{\sigma^{n+3}\omega})$ and $\psi(y_{\sigma^{n+2}\omega})=(a,y_{\sigma^{n+3}\omega})$, then replacing $n$ with $n+1$, and correspondingly swapping the letter~$z$ with $y$, we may proceed as in Subcase (a). Similarly, we proceed as in Subcase (b) if $\psi(z_{\sigma^{n+2}\omega})=(a,z_{\sigma^{n+3}\omega})$ and $\psi(y_{\sigma^{n+2}\omega})=(a,y_{\sigma^{n+3}\omega})$. Hence by our assumption, we have 
\[
\psi^{-1}\Big(\big(e, (ay_{\sigma^{n+3}\omega})^2\big)\Big) \in \varphi_w(\text{Rist}_{\GG}(w) )
\]
for $w$ an $(n+2)$nd-level vertex. Taking the commutator of this with the element in~\eqref{eq:only-z}, we obtain
\[
\psi_{n+3}^{-1} \Big(\big(e,\overset{2^{n+3}-1}\ldots,e, [z_{\sigma^{n+3}\omega},(ay_{\sigma^{n+3}\omega})^2 ]\big)\Big) \in  \text{Rist}_{\GG}(n)'. 
\]
As mentioned above, we may also assume that $\psi(z_{\sigma^{n+3}\omega})=(e,z_{\sigma^{n+4}\omega})$ and $\psi(y_{\sigma^{n+3}\omega})=(a,y_{\sigma^{n+4}\omega})$. From
\begin{align*}
    \psi\big([z_{\sigma^{n+3}\omega},(ay_{\sigma^{n+3}\omega})^2 ] \big) &= \psi\big(z_{\sigma^{n+3}\omega}y_{\sigma^{n+3}\omega} y_{\sigma^{n+3}\omega}^{\,a} z_{\sigma^{n+3}\omega} y_{\sigma^{n+3}\omega}^{\,a}y_{\sigma^{n+3}\omega} \big)\\
    &=\big( e, z_{\sigma^{n+4}\omega}y_{\sigma^{n+4}\omega} az_{\sigma^{n+4}\omega} a y_{\sigma^{n+4}\omega} \big)\\
    &=\big( e, ((z_{\sigma^{n+4}\omega}a)^2)^{y_{\sigma^{n+4}\omega}} \big),
\end{align*}
the result follows.

%\bigskip

%writing $x_{\sigma^{n+1}\omega}=y_{\sigma^{n+1}\omega}z_{\sigma^{n+1}\omega}$ and $g=[ax_{\sigma^{n+1}\omega}, (ay_{\sigma^{n+1}\omega})^2]\in \varphi_u(\text{Rist}_{\GG}(u))$, % for any $(n+1)$st-level vertex~$u$, we have
%\begin{align*}
%\psi( g^{x_{\sigma^{n+1}\omega}})&=\psi \big(y_{\sigma^{n+1}\omega}^{\,a} z_{\sigma^{n+1}\omega}y_{\sigma^{n+1}\omega}^{\,a} z_{\sigma^{n+1}\omega}\big)=\big( e, (a z_{\sigma^{n+2}\omega})^2 \big).
%\end{align*}
%Therefore 
%\[
%\big[\psi^{-1}\big((e,g^{x_{\sigma^{n+1}\omega}})\big),  (az_{\sigma^{n}\omega})^4\big]=\psi^{-1}\Big( \big(e, \big[g^{x_{\sigma^{n+1}\omega}},   (az_{\sigma^{n+1}\omega})^2 \big] \big) \Big)\in \varphi_v(\text{Rist}_{\GG}(v))'
%\]
%for any $n$th-level vertex~$v$, and {\color{blue}(there is a problem here, as I got $\big(e, (z_{\sigma^{n+2}\omega}a)^4\big)$ for the RHS.)}
%\[
%\psi\Big( \big[g^{x_{\sigma^{n+1}\omega}} , (az_{\sigma^{n+1}\omega})^2\big]\Big)=\big(e, (z_{\sigma^{n+2}\omega}a)^2\big).
%\]

%By Proposition~\ref{pr:congruence-subgroup}, we deduce that $H$ is a congruence subgroup, as required.

\medskip

For the remaining case where $\psi(z_{\sigma^n\omega})=(e,z_{\sigma^{n+1}\omega})$ and $\psi(y_{\sigma^n\omega})=(a,y_{\sigma^{n+1}\omega})$, we proceed analogously to the above, replacing $n-1$ with~$n$ in the argument.
\end{proof}

We note that the remaining Grigorchuk groups are also just infinite. Indeed, writing $k=\min\{n\in\mathbb{N}\mid \sigma^n\omega \text{ is constant}\}$ we note that $\GG/\St_{\GG}(k)$ is finite and that
$$
\St_{\GG}(k)\le G_{\sigma^k\omega} \times\overset{2^k}\ldots \times G_{\sigma^k\omega}\cong D_\infty \times\overset{2^k}\ldots \times D_\infty
$$
is a just infinite group, since $D_\infty$ is just infinite and $\GG$ acts transitively at every level of the tree.

%%%%

\section{Ramification structures for quotients of the Grigorchuk groups}

In this section, we prove Theorem~\ref{thm:main} by giving an explicit description of the ramification structures that quotients of Grigorchuk groups possess. For convenience, we will simply write $G=\GG$ for a given Grigorchuk group and we denote by $G(n)$ the quotient $G/\St_{G}(n)$ for $n\in\mathbb{N}$. In the following, by abuse of notation, we will still write $a,b_\omega, c_\omega, d_\omega$ for their respective images in~$G(n)$. Further, for $k\in\mathbb{N}$ we also write $\psi_k$ for the corresponding map when working with the quotient~$G(n)$.

\begin{proof} [Proof of Theorem~\ref{thm:main}] 
We recall $N= \max\{i_0(\sigma \omega), i_1(\sigma \omega),i_2(\sigma \omega)\} +4$. First we consider the case when $N$ is finite and we let $n\ge N$. 
We set 
$$
T_1=\{a, b_{\omega}, c_{\omega}ab_{\omega}, ad_{\omega}ab_{\omega}\}\quad\text{and}\quad T_2=\{ac_{\omega},c_{\omega},d_{\omega}ac_{\omega},ad_{\omega}ac_{\omega}\},
$$
which both yield systems of generators for~$G(n)$. We claim that the pair $(T_1,T_2)$ is  a ramification structure for $G(n)$ for $n\ge N$. To this end, we are going to show that
\begin{equation*}\label{eq:intersection}\tag{\mbox{$\mathsection$}}%\tag{\mbox{$\dagger}}
\langle x \rangle \cap \langle y \rangle^g = \{e\}
\end{equation*}
for any $x \in T_1$, $y \in T_2$, and $g \in G(n)$; in fact, 
it is enough to show that
$$
\langle x^{\nicefrac{o(x)}{2}} \rangle \cap \langle y^{\nicefrac{o(y)}{2}} \rangle^g = \{e\}.
$$

First observe that $c_\omega ab_\omega = (ad_\omega)^{c_\omega}$ and $d_\omega ac_\omega = (ab_\omega)^{d_\omega}$. Therefore, for the purpose of showing that (\ref{eq:intersection}) holds, we may assume that 
\[
T_1=\{a, b_{\omega}, ad_{\omega}, ad_{\omega}ab_{\omega}\}\quad\text{and}\quad T_2=\{ac_{\omega},c_\omega,ab_\omega,ad_{\omega}ac_\omega\}.
\]
It is easy to see that $\langle a \rangle \cap  \langle y \rangle^g = \{e\}$ for all $y \in T_2$ and $g\in G(n)$, since $a$ is not in the first level stabiliser and the elements $(ac_\omega)^2$, $c_\omega$, $(ab_\omega)^2$ and $ad_\omega ac_\omega$ are. 

Now suppose that $x=b_\omega$. 
If $y=ac_\omega$ and $o(ac_\omega)=2^{i_1+1}$, we are done since $(ac_\omega)^{2^{i_1}}\in \St_{G(n)}(i_1+1)$ from the proof of Lemma~\ref{lem: order elements}, but $b_\omega\not\in\St_{G(n)}(i_1+1)$. If $o(ac_\omega)=2^{i_1}$, then we are likewise done, as $\psi_{i_1-1}((ac_\omega)^{2^{i_1-1}})=(a,\ldots,a)$ whereas the portrait of~$b_\omega$ at the same level has at most one label with~$a$.
If $y=c_\omega$, we have that the portrait of~$c_\omega$ at level $i_2$ consists of non-trivial permutations, but the portrait of~$b_\omega$ at the same level has only trivial permutations. Hence \eqref{eq:intersection} holds here. Suppose $y=ab_\omega$. If $i_2>1$, we have $(ab_\omega)^{o(ab_\omega)/2}\in \St_{G(n)}(2)$ whereas $b_\omega\not\in\St_{G(n)}(2)$. If $i_2=1$, then writing $i=\min\{i_0,i_1\}$, we have
$\psi_i((ab_\omega)^2)$ has exactly two components consisting of the element~$a$, however $\psi_i(b_\omega)$ only has one such component. It thus follows that (\ref{eq:intersection}) holds for this choice of $x$ and $y$. Finally let $y=ad_\omega ac_\omega$ and denote its order by~$2^t$. If $i_2>\max\{i_0,i_1\}$, then (\ref{eq:intersection}) holds since $b_\omega\not\in\St_{G(n)}(2)$ but $(ad_\omega ac_\omega)^{2^{t-1}}\in\St_{G(n)}(2)$.
Likewise for the cases $1=i_0<i_2<i_1$ and $1=i_1<i_2<i_0$. Lastly, if $i_2=1$, without loss of generality suppose that $i_0<i_1$. Assume first that we are not in the case $i_1=i_0+1$ with $\omega_i=1$ for all $i_0< i\le n-1$.
%$\sigma^{i_0}\omega$ is constant. 
Then the claim follows noting that $b_\omega\notin\St_{G(n)}(i_0+1)$ 
whereas $(ad_\omega a c_\omega)^{2^{t-1}}\in\St_{G(n)}(i_0+1)$. For the remaining case, we have $i_0>2$ and
$\psi_{i_0}((ad_\omega a c_\omega)^{2^{t-1}})=(d,\overset{2^{i_0-1}}\ldots,d,a,\overset{2^{i_0-1}}\ldots,a)$, which shows that $(ad_\omega a c_\omega)^{2^{t-1}}$ cannot be conjugate to~$b_\omega$.

Next, we let $x=ad_\omega$. Similar arguments as above follow for $y=c_\omega$. So suppose $y=ac_\omega$, and without loss of generality we assume that $i_1>i_0$; the reverse case follows similarly. If $o(ac_\omega)=2^{i_1+1}$, we have $(ac_\omega)^{2^{i_1}}\in\St_{G(n)}(i_1+1)$, whereas  $(ad_\omega)^{2^{i_0}}\notin\St_{G(n)}(i_1+1)$. If $o(ac_\omega)=2^{i_1}$, then as $\psi_{i_0}((ad_\omega)^{2^{i_0}})=(d,\overset{2^{i_0}}\ldots,d)$, we are done upon comparing this with the form of $\psi_{i_1-1}((ac_\omega)^{2^{i_1-1}})$ given above. Thus (\ref{eq:intersection}) holds for this choice of~$x$ and~$y$. Likewise for $y=ab_\omega$. We consider now the remaining element $y=ad_\omega ac_\omega$. Either $i_1<i_0$ or $i_1>i_0$. In the former, writing $2^s$ for the order of~$ad_\omega$ in $G(n)$, let $j$ be minimal such that $(ad_\omega)^{2^s}\notin\St_{G(n)}(j)$. We observe that  $\psi_{j-1}((ad_\omega)^{2^s})$ has the element~$a$  appearing in both maximal subtrees, whereas for $(ad_\omega ac_\omega)^{2^{t-1}}$ and the corresponding minimal $j'$ such that  $(ad_\omega ac_\omega)^{2^{t-1}}\notin\St_{G(n)}(j')$, we have that $\psi_{j'-1}((ad_\omega ac_\omega)^{2^{t-1}})$ has~$a$ appearing only in one of the two maximal subtrees. Similarly for the case $i_1>i_0$. Hence
(\ref{eq:intersection}) holds.
%writing $j=\min\{n>i_0\mid \omega_n\ne 0\}$, if $j<\infty$ we have that $\psi_j((ad_\omega)^{2^{i_0}})$ has $2^{i_0}$ components consisting of the element~$a$ {\color{red}which appear in both maximal subtrees, whereas the portrait of $(ad_\omega ac_\omega)^{2^{t-1}}$ has labels~$a$ appearing only in one of the two maximal subtrees.} If $j=\infty$, then $d_{\sigma^{i_0-1}\omega}$ is the trivial automorphism, and thus $\psi_{i_0-1}((ad_\omega)^{2^{i_0-1}})=(a,\overset{2^{i_0-1}}\ldots, a)$. As $\psi_{i_0-1}((ad_\omega ac_\omega)^{2^{t-1}})$ has {\color{red}a maximum of} $2^{i_0-2}$ components consisting of the element~$a$, the result follows. Suppose $i_1>i_0$. If $i_0>1$ and $t>{\color{red}i_0}$, then $(ad_\omega)^{2^{i_0}}\notin\St_{G(n)}(i_1+1)$ but  $(ad_\omega ac_\omega)^{2^{t-1}}\in\St_{G(n)}(i_1+1)$. Suppose now that $t={\color{red}i_0}$, which implies that $i_1=i_0+1$, and {\color{red}$\omega_i=1$ for all $i_0< i\le n-1$.} Here the result follows since $(ad_\omega)^{2^{i_0}}\in\St_{G(n)}(i_1)$ but  $(ad_\omega ac_\omega)^{\color{red}2^{i_0-1}}\notin\St_{G(n)}(i_1)$. If $i_0=1$ and $j=2$, then  $(ad_\omega)^2\notin\St_{G(n)}(3)$ but  $(ad_\omega ac_\omega)^{2^{t-1}}\in\St_{G(n)}(3)$. Otherwise, if $j>2$, we see that both $\psi_j((ad_\omega)^2)$ and $\psi_j((ad_\omega ac_\omega)^{2^{t-1}})$ have $2$ components consisting of the element~$a$, however in $\psi_j((ad_\omega)^2)$ these elements are in separate maximal subtrees, and in $\psi_j((ad_\omega ac_\omega)^{2^{t-1}})$ they are in the same maximal subtree. Hence (\ref{eq:intersection}) holds.

Finally we set $x=ad_\omega ab_\omega$. Likewise, arguments as above show that (\ref{eq:intersection}) holds for $y\in\{ac_\omega, c_\omega ,ab_\omega\}$. Hence we suppose $y=ad_\omega a c_\omega$. 
Then (\ref{eq:intersection}) follows from considering the components of $\psi(ad_\omega ab_\omega)$ and $\psi(ad_\omega ac_\omega)$. We note however, that if $i_0=1$, then $\psi((ad_\omega ab_\omega)^2)=((d_{\sigma\omega}a)^2,e)=\psi((ad_\omega ac_\omega)^2)$. In this case, we consider $
T_1=\{a, d_{\omega}, c_{\omega}ad_{\omega}, ab_{\omega}ad_{\omega}\}$ and  $T_2=\{ac_{\omega},c_{\omega},b_{\omega}ac_{\omega},ab_{\omega}ac_{\omega}\}$, and proceed analogously.

\smallskip

We now consider the non-periodic Grigorchuk groups~$G=\GG$ with infinite~$N$ such that $\sigma\omega$ is not a constant sequence, and we use the same $T_1$ and $T_2$.
As before, it is clear that $\langle a \rangle \cap  \langle y \rangle^g = \{e\}$ for all $y \in T_2$ and $g\in G(n)$. 
To proceed, we distinguish two cases.

%\smallskip

\underline{Case 1}: Suppose first that $i_0,i_1,i_2$ and two of the  $i_0(\sigma\omega), i_1(\sigma\omega), i_2(\sigma\omega)$ are finite. Without loss of generality, suppose  $i_2(\sigma\omega)$ is infinite. This implies that $w_1=2$,  $w_i\in\{0,1\}$ for all $i\ge 2$, and $i_0=i_0(\sigma\omega)+1$, $i_1=i_1(\sigma\omega)+1$. Let $n\ge \widetilde{N}=\max\{i_0,i_1\}+3$. 

Let $x=b_\omega\in\St_{G(n)}(2)\backslash \St_{G(n)}(3)$. 
Suppose $y=ac_\omega$. Here we have two situations to consider. %for working out the order of $ac_\omega$ in~$\GG(n)$. 
If $\omega_i=1$ for all $3\le i_1\le i\le n-1$, then $(ac_\omega)^{2^{i_1}}=\{e\}$ in~$G(n)$, as $\psi_{i_1-1}((ac_\omega)^{2^{i_1-1}})=(a,\overset{2^{i_1-1}}\ldots,a)$ in~$G(n)$. Otherwise $o(ac_\omega)=2^{i_1+1}$. In the latter case
we are done since $(ac_\omega)^{2^{i_1}}\in \St_{G(n)}(3)$. In the former case, if $i_1>3$, we are likewise done. So assume $i_1=3$.
Here $\psi_2((ac_\omega)^4)=(a,a,a,a)$, which clearly cannot be conjugate to $\psi_2(b_\omega)$. Hence \eqref{eq:intersection} holds here.
If $y=c_\omega$, we are done similarly since $c_\omega\notin\St_{G(n)}(2)$.  Suppose $y=ab_\omega$. Here
$\psi_2((ab_\omega)^2)$ has exactly two components consisting of~$a$, however $\psi_2(b_\omega)$ only has one such component. It thus follows that (\ref{eq:intersection}) holds for this choice of $x$ and $y$. Finally let $y=ad_\omega ac_\omega$ and denote its order as before by~$2^t$. The claim follows upon recalling that $\psi_2(b_\omega)$ has exactly one component consisting of~$a$, whereas this is not the case for $\psi_2((ad_\omega a c_\omega)^{2^{t-1}})$. Indeed, if we are not in the case $i_0=2$, $\omega_i=1$ for $3\le i\le n-1$ or $i_1=2$, $\omega_i=0$ for $3\le i\le n-1$, then $(ad_\omega a c_\omega)^{2^{t-1}}\in\St_{G(n)}(3)$. In the remaining cases, we have $\psi_2((ad_\omega a c_\omega)^{2^{t-1}})=(d,d,a,a)$ or $\psi_2((ad_\omega a c_\omega)^{2^{t-1}})=(a,a,c,c)$ respectively, and the result is clear.

Next, we let $x=ad_\omega$. The case $y=c_\omega$ is clear, and for $y=ac_\omega$ we are done as in the periodic case above.
For $y=ab_\omega$, the claim follows from noting that $(ad_\omega)^{2^{i_0}}\in\St_{G(n)}(3)$, if  $o(ad_\omega)=2^{i_0+1}$ in~$G(n)$, whereas $(ab_\omega)^2\notin\St_{G(n)}(3)$. If $o(ad_\omega)=2^{i_0}$ in~$G(n)$, then we are likewise done unless $i_1=2$, and $\omega_i=0$ for $3\le i\le n-1$. In this final case, we have $\psi_2((ad_\omega)^{2^{i_0-1}})=(a,a,a,a)$ and the result is clear. We consider now the remaining element $y=ad_\omega ac_\omega$. Suppose first that $i_1=2$. Then for some $i\ge i_0-1$, writing $2^s$ for the order of $ad_\omega$ in~$G(n)$, we have $\psi_{i}((ad_\omega)^{2^{s-1}})$ has components consisting of the element~$a$ stemming from both maximal subtrees but %either $(ad_\omega ac_\omega)^{2^{t-1}}\in\St_{G(n)}(i+1)$ or 
$\psi_{i}((ad_\omega ac_\omega)^{2^{t-1}})$  has components consisting of~$a$ stemming from only one maximal subtree. Similarly for the case  $i_0=2$. %Now we suppose that $i_0=2$. We are done since either, if $\omega_i=0$ for some $i_1+1\le i \le n-1$, we have $(ad_\omega)^4\notin\St_{G(n)}(i_1+1)$ but $(ad_\omega a c_\omega)^{2^{t-1}}\in\St_{G(n)}(i_1+1)$, or if $\omega_i=1$ for all $i_1\le i \le n-1$ then $(ad_\omega)^4\in\St_{G(n)}(i_1)$ but $(ad_\omega a c_\omega)^{2^{t-1}}\notin\St_{G(n)}(i_1)$.

Finally we set $x=ad_\omega ab_\omega$ and consider $y=ac_\omega$. Let us write $2^v$ for the order of $ac_\omega$. We note that $o(x)=2^{n-1}$ in~$G(n)$. Further  $(ad_\omega ab_\omega)^{2^{n-2}}\in\St_{G(n)}(n-1)$ and $\psi_{n-1}((ad_\omega ab_\omega)^{2^{n-2}})=(e,\overset{2^{n-2}}\ldots, e,a,\overset{2^{n-2}}\ldots, a )$. However, for some $3\le i\le n-1$ we have $(ac_\omega)^{2^{v-1}}\notin\St_{G(n)}(i)$, hence \eqref{eq:intersection} holds. A similar argument follows for $y=ab_\omega$. The next choice of $y=c_\omega\notin\St_{G(n)}(2)$ is straightforward, so it remains to consider 
 $y=ad_\omega a c_\omega$. As
  $(ad_\omega a c_\omega)^{2^{t-1}}\notin\St_{G(n)}(n-1)$, the result follows.
  
  \underline{Case 2}: Suppose  that two of the $i_0,i_1,i_2$ and two of the $i_0(\sigma\omega), i_1(\sigma\omega), i_2(\sigma\omega)$ are finite. Without loss of generality, suppose   $i_2$ is  infinite. This implies that $w_i\in\{0,1\}$ for all $i\in \mathbb{N}$. Further we may assume that $i_1=1$ and then $i_0=i_0(\sigma\omega)+1$. The other case is analogous. 

Let $n\ge \widetilde{N}=\max\{i_0(\sigma\omega),i_1(\sigma\omega)\}+4$, and let 
$x=b_\omega\in\St_{G(n)}(1)\backslash \St_{G(n)}(2)$. If $y=ac_\omega$, we are done since $(ac_\omega)^2\in\St_{G(n)}(2)$.
If $y=c_\omega$, we notice that $\psi(c_\omega)=(e,c_{\sigma\omega})$ whereas $\psi(b_\omega)=(a,b_{\sigma\omega})$.
Hence \eqref{eq:intersection} holds. Suppose $y=ab_\omega$ and note that $o(ab_\omega)=2^n$ in~$G(n)$. As $(ab_\omega)^{2^{n-1}}\in\St_{\GG(n)}(2)$, we are done. Finally let $y=ad_\omega ac_\omega$ and write $2^t$ for its order in~$G(n)$. Here we also have $(ad_\omega a c_\omega)^{2^{t-1}}\in\St_{G(n)}(2)$.

Next, we let $x=ad_\omega$. Here $o(ad_\omega)$ is either $2^{i_0+1}$ or $2^{i_0}$ in~$G(n)$, with the latter happening only when $\omega_i=0$ for all $i_0\le i\le n-1$. Suppose $y=c_{\omega}$. If $o(ad_\omega)=2^{i_0+1}$, then $(ad_\omega)^{2^{i_0}}\in\St_{G(n)}(i_0+1)$ whereas $c_\omega\notin\St_{G(n)}(i_0+1)$. In the remaining case, we have $(ad_\omega)^{2^{i_0-1}}\notin\St_{G(n)}(i_0)$ but $c_\omega\in\St_{G(n)}(i_0)$. Hence
\eqref{eq:intersection} holds. % since $c_{\omega} \notin \St_{\GG(n)}(2)$. 
For $y=ac_{\omega}$, we proceed likewise. For $y=ab_{\omega}$, we have $\psi_{n-1}((ab_\omega)^{2^{n-1}})=(a,\overset{2^{n-1}}\ldots, a)$, whereas, writing $2^s$ for the order of $ad_\omega$ in~$G(n)$, we have $(ad_\omega)^{2^{s-1}}\notin\St_{G(n)}(n-1)$.
Finally let $y=ad_{\omega}ac_{\omega}$. If $i_0\ge 3$, then $(ad_\omega ac_\omega)^{2^{t-1}}\in\St_{G(n)}(i_0)\backslash \St_{G(n)}(i_0+1)$, whereas, if $o(ad_\omega)=2^{i_0+1}$ then $(ad_{\omega})^{2^{i_0}}\in\St_{G(n)}(i_0+1)$, and if $o(ad_{\omega})=2^{i_0}$ in~$G(n)$, then $(ad_{\omega})^{2^{i_0-1}}\notin\St_{G(n)}(i_0)$. If $i_0=2$, we have $(ad_\omega)^4\in \St_{G(n)}(i_1(\sigma\omega)+1)\backslash \St_{G(n)}(i_1(\sigma\omega)+2)$ with $\psi_{i_1(\sigma\omega)+1}((ad_\omega)^4)$ having exactly four components consisting of~$a$, where these components occur in both maximal subtrees. Now for $(ad_\omega ac_\omega)^{2^{t-1}}$,  there exists a minimal $i\ge i_1(\sigma\omega)$ such that $\psi_i((ad_\omega ac_\omega)^{2^{t-1}})$ has components consisting of~$a$. However these components only occur in the right maximal subtree.

Consider now the remaining case $x=ad_\omega ab_\omega$. First note that $o(x)=2^{n-1}$ in~$G(n)$. Also $(ad_\omega ab_\omega)^{2^{n-2}}\in\St_{G(n)}(n-1)$ and $\psi_{n-1}((ad_\omega ab_\omega)^{2^{n-2}})=(e,\overset{2^{n-2}}\ldots, e,a,\overset{2^{n-2}}\ldots, a )$.  If $y=ac_\omega$, as
$(ac_\omega)^{2}\in \St_{G(n)}(i_0) \backslash \St_{G(n)}(i_0+1)$, the claim \eqref{eq:intersection} holds. A similar argument follows for $y=c_\omega$. Let $y=ab_\omega$. 
As $\psi_{n-1}((ab_\omega)^{2^{n-1}})=(a,\overset{2^{n-1}}\ldots, a )$, the claim follows. 
Now it remains to consider 
 $y=ad_\omega a c_\omega$. As
  $(ad_\omega a c_\omega)^{2^{t-1}}\notin\St_{G(n)}(n-1)$, we are done.

\smallskip

We now prove the last part of the theorem. Note that this then covers all remaining Grigorchuk groups as the case where $i_0, i_1,i_2$ and one of the  $i_0(\sigma\omega), i_1(\sigma\omega), i_2(\sigma\omega)$ are finite cannot exist.

Without loss of generality suppose 
that $w_1=2$ and  $w_i=0$ for all $i\ge 2$. Let $n\ge 4$
and set
$$
T_1=\{a, b_{\omega}, c_{\omega},ad_{\omega}\}\quad\text{and}\quad T_2=\{ab_{\omega},d_{\omega},ab_{\omega}ad_{\omega},ac_{\omega}ab_{\omega}ad_{\omega}\}.
$$
The result is clear for $x=a$, so let  
$x=b_\omega\in\St_{G(n)}(2)\backslash \St_{G(n)}(3)$. For $y=ab_\omega$, we observe that $o(ab_\omega)=4$ in~$G(n)$ and $\psi((ab_\omega)^2)=(b_{\sigma\omega},b_{\sigma\omega})$, whereas $\psi(b_\omega)=(e,b_{\sigma\omega})$. So the result is clear. Next let  $y=d_\omega\in\St_{G(n)}(1)\backslash \St_{G(n)}(2)$, which is also straightforward.   Let $y=ab_\omega ad_{\omega}$, which has order~$2^{n-1}$ in~$G(n)$. Further $(ab_\omega ad_{\omega})^{2^{n-2}}\in\St_{G(n)}(n-1)$, hence the result.
 For $y=ac_{\omega}ab_{\omega}ad_{\omega}$, since $d_{\sigma\omega}=e$ we have
 $\psi\big((ac_{\omega}ab_{\omega}ad_{\omega})^2\big)=(c_{\sigma\omega}ab_{\sigma\omega}a,ab_{\sigma\omega}ac_{\sigma\omega})$. Thus $o(ac_{\omega}ab_{\omega}ad_{\omega})={2^{n-1}}$ in~$G(n)$ and as $(ac_{\omega}ab_{\omega}ad_{\omega})^{2^{n-2}}\in\St_{G(n)}(n-1)$,  the result follows.

Let $x=c_\omega\in\St_{G(n)}(1)\backslash \St_{G(n)}(2)$. For all $y\in T_2$, we are done using the above information; that is, either  $y^{\nicefrac{o(y)}{2}}\in \St_{G(n)}(2)$, or $y=d_\omega$ and since $\psi(d_\omega)=(a,e)$, the result follows.

Lastly, for $x=ad_\omega$,  we have $o(ad_\omega)=4$, and $\psi((ad_\omega)^2)=(a,a)$. Again we are done using the above information.
\end{proof}

%%%%%%%%%%%%%%%%%

\end{document}